\documentclass[12pt]{amsart}

\theoremstyle{plain}

\numberwithin{equation}{section}

\usepackage{comment}
\usepackage{amsmath}
\usepackage{fancybox,color}
\usepackage{latexsym}
\usepackage{mathrsfs}
\usepackage{textcomp}
\usepackage{graphicx}

\usepackage{pxfonts}
\usepackage{esint}
\usepackage{textcomp}
\usepackage{vmargin}
\usepackage{bbm}
\usepackage[active]{srcltx}

\usepackage{vmargin}

\setmarginsrb{30mm}{20mm}{30mm}{20mm}{10mm}{10mm}{10mm}{10mm}

\numberwithin{equation}{section}

\newtheorem{theo}{Theorem}[section]
\newtheorem{lem}[theo]{Lemma}

\newtheorem{cor}[theo]{Corollary}

\newtheorem{defi}[theo]{Definition}
\newtheorem{rem}[theo]{Remark}

\input{epsf}

\def\C{{\rm\kern.24em \vrule width.02em height1.4ex
depth-.05ex\kern-.26em C}}
\def\N{{\rm I\hspace{-0.4ex}N}}
\def\R{{\rm I\hspace{-0.4ex}R}}
\def\B{{\rm I\hspace{-0.4ex}B}}
\def\D{{\rm I\hspace{-0.4ex}D}}

\def\be{\begin{equation}}
\def\ee{\end{equation}}
\def\beq{\begin{equation}}
\def\eeq{\end{equation}}

\input{epsf}

\begin{document}

     \title{A note on $p$-harmonic measure}

\author[J. G. Llorente]{Jos\'{e} G. Llorente}
\address{Departamento de Análisis Matemático y Matemática Aplicada. Facultad de Matemáticas. Universidad Complutense de Madrid. Plaza de Ciencias 3, 28040 Madrid (Spain)}
\email{josgon20@ucm.es}



\thanks{Partially supported by grant PID 2021-123151NB-I00 (Spain).}

\maketitle

\begin{abstract}
Let $1 < p < \infty$. We identify the optimal decay exponent of the $p$-harmonic measure $\omega_p ( B(\xi, t) \cap \partial \B , 0, \B ) $ for small $t$, where $\B$ is the unit ball in $\R^n$, $\xi \in \partial \B$ and $B(\xi,t)$ denotes the euclidean ball centered at $\xi$ of radius $t$. As a consequence and, under suitable geometrical restrictions on a domain $\Omega \subset \R^n$, we identify the optimal decay exponent of $\omega_p ( B( \xi, t) \cap \partial \Omega, a, \Omega )$ for small $t$, where $a \in \Omega$ and $\xi \in \partial \Omega$.
When restricted to the range $1 < p < n$ we provide new estimates complementing previous results of different authors in the unit ball, the unit disc and the upper half-space.
\end{abstract}
\noindent


\section{Introduction}\label{Intro}

If $\Omega \subset \R^n$ is a domain, $E\subset \partial \Omega$ and $x\in \Omega$ then the \textit{harmonic measure} of $E$ in $\Omega$ from $x$, denoted $\omega(E, x, \Omega)$, is the value at $x$ of the harmonic function in $\Omega$ with boundary values $1$ on $E$ and $0$ on $\partial \Omega \setminus E$, provided $\Omega$ and $E$ are sufficiently regular. From the linearity of the Laplace operator it follows that $\omega ( . \, , x, \Omega )$ is a probability measure on $\partial \Omega$  for each $x\in \Omega$ and Harnack property implies that if $x, y \in \Omega$ then $\omega ( . \, ,x, \partial \Omega )$ and $\omega ( . \, ,y, \partial \Omega)$ are mutually absolutely continuous. From a probabilistic point of view, $\omega (E,x, \Omega) $ is the probability that brownian motion starting at $x$ hits $\partial \Omega$ for the first time at $E$. The study of the metric properties of harmonic measure and its connection to the geometry of  $\partial \Omega $ plays a fundamental
role in the interplay of  Geometric Function Theory, Potential Theory, PDE's, Probability and Dynamical Systems, among other fields.

\

Harmonic functions in a domain minimize the $2$-energy of the gradient with prescribed boundary values and the Laplace equation is the corresponding Euler-Lagrange equation. If $1< p< \infty$, the Euler-Lagrange equation associated to  minimization of the $p$-energy is the non-linear $p$-harmonic equation:
\begin{eqnarray}\label{plaplace}
\triangle_p  u \equiv \textrm{div}\left( |\nabla u |^{p-2} \nabla u \right) = 0  \, , \, \, \, \, \, \, \, (1 < p
< \infty ).
\end{eqnarray}
where $\triangle_p$ is the so called $p$-\textit{laplacian} operator and weak solutions of \eqref{plaplace} in the Sobolev space
$W^{1,p}_{\textrm{loc}}(\Omega )$ are called \textit{$p$-harmonic} functions in $\Omega$. Note that $\triangle_2 $ is the Laplace operator so the classical case is recovered when $p = 2$. We refer to \cite{Li}, \cite{HKM} for basic properties of $p$-harmonic functions and general background on nonlinear potential theory.

\

The formal limit of $\triangle _p $ when $p \to \infty $ is the so called
\textit{infinity laplacian} $\triangle_{\infty}$, a
(non-divergence) differential operator given by
\begin{eqnarray}
\triangle_{\infty}u =  \sum_{i,j=1}^N u_{x_i}u_{x_j} u_{x_i x_j}
\label{infty}
\end{eqnarray}
and viscosity solutions of $\triangle_{\infty}u = 0$
are called \textit{infinity harmonic} functions.

\

If $p\not=2$, the definition of $p$-harmonic measure $\omega_p (. \, ,
x, \Omega ) $ mimics the previous potential theoretic approach in the linear case $p = 2$ (see
\cite{HKM} for technical details). However, because of the nonlinearity, $p$-harmonic measures are
more difficult to work with  and lack some  nice properties of the case $p=2$. First, $p$-harmonic measure is no longer a measure if $p\neq 2$, not even at the zero level (see \cite{LLMW}). Second, if $\Omega = \mathbb{\B}$ (the unit ball in $\R^n$) or  $\R^n_{+}$ (the upper half space in $\R^n$), $E \subset \partial \Omega$, $a \in \Omega$ and $p=2$ then it is easy to estimate $\omega(E,a,\Omega)$ by using the explicit expression of the Poisson kernel in the unit ball or the upper half-space. The situation is much  more complicated
when $p\neq 2$, even if $n=2$ and $E$ is a simple subset of the unit circle, the unit sphere or the real line, like an arc, a spherical cap or an interval. Nevertheless, having estimates of $\omega_p ( B(\xi , t) \cap \partial \Omega , a, \Omega )$ for small $t$, where $a\in \Omega$ and $\xi \in \partial \Omega $ is a relevant problem in Nonlinear Potential Theory. Hereafter $\B$ and $\D$ denote the unit ball in $\R^n$ and the unit disc, respectively, $B( \xi, t)$ stands for the euclidean ball centered at $\xi$  of radius $t$ and $|x-y|$ denotes the euclidean distance between $x$ and $y$ in $\R^n$.

 \begin{defi}\label{defpower}
Let $\Omega \subset \R^n $ be a domain and $1 < p < \infty$.
\begin{enumerate}
\item[a)] We say that $\Omega$ satisfies a $p$-\textbf{Upper Power Law} of exponent $q>0$ if
 \begin{equation*}\label{uppowerlaw}
 \limsup_{t\to 0} \frac{\omega_p ( B(\xi , t) \cap \partial \Omega , a, \Omega )}{ t^q} < \infty
 \end{equation*}
for all $a\in \Omega$ and $\xi \in \partial \Omega$. \vspace{0.12cm}
\item [b)] We say that $\Omega$ satisfies a $p$-\textbf{ Lower Power Law} of exponent $q>0$ if
 \begin{equation*}\label{lowpowerlaw}
 \liminf_{t\to 0} \frac{\omega_p ( B(\xi , t) \cap \partial \Omega , a, \Omega )}{ t^q} > 0
 \end{equation*}
for all $a\in \Omega$ and $\xi \in \partial \Omega$. \vspace{0.12cm}
\item[c)] We say that $\Omega$ satisfies a $p$-\textbf{Power Law} of exponent $q>0$ if it satisfies both the upper and lower law with exponent $q$ or, equivalently, if
    \begin{equation*}\label{powerlaw}
0 <  \liminf_{t\to 0} \frac{\omega_p ( B(\xi , t) \cap \partial \Omega , a, \Omega )}{ t^q} \leq  \limsup_{t\to 0} \frac{\omega_p ( B(\xi , t) \cap \partial \Omega , a, \Omega )}{ t^q} < \infty
    \end{equation*}
for all $a\in \Omega$ and $\xi \in \partial \Omega$.
\end{enumerate}
\end{defi}

\begin{defi}\label{defasypower}
We say that $\Omega$ satisfies a $p$-\textbf{Asymptotic Power Law} of exponent $q>0$  if
\begin{equation*}
\lim_{t\to 0} \frac{\log \omega_p ( B(\xi , t) \cap \partial \Omega , a, \Omega )}{\log t} = q
\end{equation*}
for all $a\in \Omega$ and $\xi \in \partial \Omega$.
\end{defi}

\subsection{Previous results}\label{prevsub}

One of the first results giving precise asymptotic estimates for $p$-harmonic measure was obtained by Peres, Schramm, Sheffield and Wilson
(\cite{P}) in the case $p = \infty$. They proved that $\B$ (the unit ball of $\R^n$), satisfies a $\infty$- Power Law of exponent $1/3$, that is,
\begin{eqnarray*}
\omega_{\infty} (B(\xi, t) \cap \partial \B , 0,\B ) \approx t^{1/3}
\label{inftymeas}
\end{eqnarray*}

\

As for the case $1 < p < \infty$, Lundstr\"om and Vasilis (\cite{Lu}) obtained sharp estimates for $\omega_p (E,a, \Omega)$ when $n=2$ and $\Omega$ is a planar domain satisfying certain smoothness assumptions. If $1 < p < \infty$, let
\begin{equation}\label{q}
q =\frac{3-p + 2\sqrt{p^2 -3p +3}}{3(p-1)}
\end{equation}

By specializing to the cases when $\Omega = \D$ (the unit disc) or $\Omega = \R^2_+$ (the upper half-plane), it is proved in \cite{Lu} that

\begin{eqnarray}
& \hspace{-0.2cm} \R^2_+ & \text{satisfies a} \, p\text{-Power Law of exponent} \, q \, \text{for any} \,  p \in (1, \infty ). \label{phalf}   \vspace{0.2cm} \\
& \hspace{-0.12cm} \D & \text{satisfies a } \, p\text{-Power Law of exponent} \, q \, \text{for any} \, p \in [2, \infty ). \label{pdisc} \vspace{0.12cm} \\
& \hspace{-0.2cm} \D & \text{satisfies a} \,  p\text{-Upper Power Law of exponent} \, q \,  \text{for any} \, p \in (1,2]. \label{pdiscup}
\end{eqnarray}
where $q = q(p)$ is given by \eqref{q}.

\

One of the main technical tools needed in \cite{P} and \cite{Lu}  is the use of Aronsson's \textit{singular quasiradial $p$-harmonic functions} (\cite{Ar}) which play, somehow, the role of the Poisson kernel for $p=2$. Recall that the Poisson kernel in $\R^n_+$ with pole at $a \in \R^{n-1}$ is given by

\begin{equation}\label{Poisson}
P(z,a) = \frac{y}{[(x-a)^2 + y^2]^{n/2}}
\end{equation}
where $z = (x,y) \in \R^n_+$, $x \in \R^{n-1}$ and $y>0$. If $a = 0$, $r =  |z| = (|x|^2 + y^2)^{1/2}$ and $\alpha$ is the latitude angle between $z$ and the hyperplane $\R^{n-1}$ then we get the expression of  the Poisson kernel with pole at the origin in polar coordinates:
$$
P(z) \equiv P(z,0) = r^{1-n} \sin \alpha
$$
so the singularity at $0$ is of order $r^{1-n}$. Note that the Poisson kernel is harmonic and positive in $\R^n_+$, has boundary values $0$ at $\R^{n-1} \setminus \{ 0\}$ and blows up nontangentially to $\infty $ at $0$. If $p\not=2$, the analogue role of the Poisson kernel in $\R^n_+$ is played
by singular quasiradial functions of the form $h = r^k f(\alpha )$ with $k<0$, where $\alpha$ denotes the latitude angle as above. When using polar coordinates it follows that $\triangle_p h $ is given (up to a positive factor) by the left hand side in the following nonlinear ODE:
\begin{eqnarray}\label{ode}
[(p-1)(f')^2 + k^2 f^2]f'' + k[(2p-3)k + n-p]f(f')^2 +  \nonumber \\ k^3[k(p-1) + n-p]f^3 - (n-2)[(f')^2 + k^2 f^2 ]f' \tan \alpha = 0
\end{eqnarray}

\

\noindent (See \cite{LLMTW}, \cite{deb}, \cite{deb2} and note that, since those  authors use the co-latitude angle $ \theta = \pi /2 - \alpha$ instead of $\alpha$, there is a sign change in the four summand, coming from $f'$; see equation (1.11) in \cite{LLMTW}). The key point is that for each $p$ with $1 < p < \infty$ there exists a unique value of $k = k(p,n) <0$ and $f: [0, \pi/2] \to \R \in C^2 [0, \pi/2]$ verifying
\begin{eqnarray}
& f(0)  = 0, \, f(\pi/2) = 1, \,  f'(\pi /2) = 0 \,  \textnormal{and} \, \, 0
< f'(0) < \infty. \label{f1} \\
& f  \, \textnormal{is increasing in} \,  (0, \pi /2)  \label{f2}
\end{eqnarray}
and such that $f$  satisfies \eqref{ode}, so $h = r^k f( \alpha)$ is positive and $p$-harmonic in $\R^n_+$ (see \cite{LN}, \cite{PV} , \cite{T} for PDE approaches and \cite{LLMTW} for a purely ODE proof based on shooting techniques).  Then $h$ has boundary value $0$ at $\R^{n-1} \setminus \{ 0\}$ and a singularity of order $r^{k}$ at $0$. We will refer to $h$ as the \textit{singular quasiradial} $p$-harmonic function in the half-space $\R^n_+$ with pole at $0$. Note that if $n=2$ then the last term in \eqref{ode} disappears and the corresponding ODE becomes autonomous and simpler to work with. It turns out that $k(p,2) = -q$, where $q$ is given by \eqref{q} (see \cite{Ar}, \cite{Lu}). On the other hand, if $n\geq 3$ then \eqref{ode} is much more difficult to handle and the critical value $k(p,n)$ cannot be explicitly computed.  We summarize some specific cases where the critical exponent $k(p,n)$ is known:

\begin{itemize}
\item $k(2,n) = -(n-1)$ (Poisson kernel). \vspace{0.12cm}
\item $ \displaystyle  k(p,2) = -q = - \frac{3-p + 2\sqrt{p^2 -3p +3}}{3(p-1)} $ (\cite{Ar}, \cite{Lu}). \vspace{0.12cm}
\item $k(n,n) = -1$ (the conformal case, see  \cite{Hi}). \vspace{0.12cm}
\item $k(\infty , n) = - \frac{1}{3}$ (see \cite{P}).
\end{itemize}

\

Furthermore $k(p,n) \leq -1$ if $1 < p  \leq n$ and $k(p,n) \geq -1$ if $p \geq n$ (see the Appendix). The interest in singular quasiradial functions relies on the following fact:  by means of comparison arguments, the $p$-superharmonicity (resp. $p$-subharmonicity) of the singular quasiradial function $h = r^k f( \alpha)$ in $\R^n_+$ implies a $p$-Upper Power Law  of exponent $q = - k$ (resp. a $p$-Lower Power Law) for $p$-harmonic measure in $\R^n_+$ (see \cite{LLMTW}, section 2). Obtaining upper or lower laws for $p$-harmonic measure in the half-space then reduces to studying the $p$-superharmonicity or $p$-subharmonicity of such singular quasiradial functions. In particular, $\R^n_+$ satisfies a $p$-Power Law of exponent $q = -k(p,n)$ (Corollary 1.3 in \cite{LLMTW}). See also \cite{LLMTW} and \cite{deb2} for upper and lower estimates of $k$  implying directly upper and lower laws for $p$-harmonic measure in $\R^n_+$, depending on different ranges of $p$.

\

\subsection{Main results}\label{mainsub}

We start by discussing the case of the unit ball. Let $t>0$ be small. By rotational invariance, it is enough to consider spherical caps on $\partial \B$  centered at the point $\mathbf{1} = (1, 0,...,0)$. Furthermore, it will be convenient to replace $B(\mathbf{1},t) \cap \partial \B$ by the spherical cap $J(t)$ consisting of those points on $\partial \B$ forming an angle not greater than $t$ with $\mathbf{1}$; more precisely:
\begin{equation}\label{defJ}
J(t) = \{ x\in \partial \B : \, < x, \mathbf{1} > \, \geq \cos t \}
\end{equation}
where $<\cdot , \cdot >$ stands for the usual dot produt in $\R^n$. Since $B(\mathbf{1},t) \cap \partial \B = J (\arcsin t)$, such reduction does not affect our estimates. Note also that $J(t) = \{ e^{i\theta}: \, |\theta| \leq t   \}$ if $n = 2$.

\

The extension of the $p$-Power Law in $\R^n_+$ (recall Corollary 1.3 in \cite{LLMTW}) to the ball $\B$ is by no means direct, mainly because of the lower law. Still in the case of the unit disc $\D$, observe that  $\eqref{pdiscup}$ provides only the upper law in the range $1 < p < 2$, so there is a gap for the lower law. The reason of this gap is, roughly speaking, that the lower law requires the control of the singular $p$-harmonic quasiradial function $h = r^k f( \alpha)$ on the boundary of certain subdomains of $\D$. It turns out that if $1 < p < 2$, the level curves of $h$ are flatter than $\partial \D$, which causes technical difficulties when performing comparison arguments. This obstacle is not present if $p\geq 2$ or in $\R^n_+$ (the boundary is totally flat). The way we will circumvent such an obstacle is to obtain an adjusted comparison result (Lemma \ref{majorilem} below). We refer to the beginning of section  \ref{majori} for further details.

\

The main goal of this note is to extend Lundstr\"{o}m-Vasilis results (in the case of $\D$) to the unit ball $\B$ and the range $1 < p <n$. As a consequence, we will deduce the same results for more general domains. We recall that $k(p,n)$ is the critical exponent described in the previous section. The next theorem is the main result of the note.

\begin{theo}\label{mainth}
There exist $0 < t_0 < 1$ and $\beta , b_1 , b_2 >0 $, only depending on $p$ and $n$, such that, if $0 < t \leq t_0$,
\begin{eqnarray}
 b_1 t^q \leq \omega_p (J(t), 0, \B ) \leq b_2 \, t^q & \quad \text{if} \, \, & p \geq n \label{mainth1} \\
 b_1 \frac{t^q}{ \big ( \log \frac{1}{t} \big )^{\beta}} \leq \omega_p (J(t), 0, \B ) \leq b_2 \, t^q & \quad \text{if} \, \, & 1 < p < n \label{mainth2}
\end{eqnarray}
where $q = -k(p,n)$.
\end{theo}

\begin{cor}\label{maincor}
There exist $0 < t_0 < 1$, $\beta , b_1 , b_2 >0 $, only depending on $p$ such that, if $0 < t \leq t_0$ then
\begin{eqnarray}
 b_1 t^q \leq \omega_p (J(t), 0, \D ) \leq b_2 t^q & \quad \text{if} \, \, & p \geq 2 \label{maincor1} \\
b_1 \, \frac{t^q}{ \big ( \log \frac{1}{t} \big )^{\beta}} \leq \omega_p (J(t), 0, \D ) \leq b_2 t^q  & \quad \text{if} \, \, & 1 < p < 2 \label{maincor2}
\end{eqnarray}
where $q = -k(p,2)$ is given by \eqref{q}.
\end{cor}

From  Theorem \ref{mainth} we deduce:
\begin{cor}\label{maincorqual} \
\begin{enumerate}
\item [a)] If $p \geq n$ then $\B$ satisfies a $p$-Power Law with exponent $q = -k(p,n)$. \vspace{0.12cm}
\item [b)] If $1 < p < n$ then $\B$ satisfies a $p$-Upper Power Law and a $p$-Asymptotic Power Law of exponent $q = -k(p,n)$.
\end{enumerate}
In particular, $\D$ satisfies a $p$-Asymptotic Power Law with exponent $q$ if $1 < p < \infty$, where $q = -k(p,2)$ is given by  \eqref{q}.
\end{cor}

\begin{rem}\label{mainrem}\rm
The novelty of Theorem \ref{mainth} relies mainly on the lower bound in \eqref{mainth2} and the result seems to be new even in $\D$. \eqref{maincor1}  and the upper bound in \eqref{maincor2} are contained in \cite{Lu}. The upper estimates in \eqref{mainth1} and \eqref{mainth2} are contained in Theorem 2.1 in \cite{LLMTW}.
\end{rem}

\begin{rem} \rm The logarithmic term in Theorem \ref{mainth} might be a consequence of the method of proof and could perhaps be avoided  by using  more sophisticated potential theoretical tools. On the other hand, our argument is elementary, clean and provide the asymptotic optimality of the exponent $q$.
\end{rem}

\

Theorem \ref{mainth} together with straightforward comparison arguments imply corresponding results for domains satisfying certain geometrical restrictions (see Theorem  \ref{genthm} below). Recall that $B(x,r)$ denotes the (open) ball centered at $x$ of radius $r$ in $\R^n$.


\begin{defi}\label{inball}
We say that the domain $\Omega \subset \R^n$ satisfies the \textbf{Interior Ball Condition} if for each $\xi \in \partial \Omega$ there exists $b\in \Omega$ such that $B(b, |b- \xi |) \subset \Omega$
\end{defi}

\begin{theo}\label{genthm}
Let $\Omega \subset \R^n$ be a convex domain satisfying the Interior Ball Condition, $a\in \Omega$, $\xi \in \partial \Omega$ and $q = -k(p,n)$.
\begin{enumerate}
\item [a)] If $p \geq n$ then
$$
0 < \liminf_{t\to 0} \frac{\omega_p (B( \xi, t) \cap \Omega , a, \Omega )}{t^q} \leq \limsup_{t\to 0} \frac{\omega_p (B( \xi, t) \cap \Omega , a, \Omega )}{t^q} < \infty
$$
In particular, $\Omega$ satisfies a $p$-Power Law with exponent $q$. \vspace{0.12cm}
\item [b)] If $1 < p < n$ then there exists $\beta = \beta (p,n) >0$ such that
$$
0 < \liminf_{t\to 0} \frac{\omega_p (B( \xi, t) \cap \Omega , a, \Omega )}{t^q \big ( \log \frac{1}{t} \big )^{- \beta} } \, \,  \text{and} \, \, \limsup_{t\to 0} \frac{\omega_p (B( \xi, t) \cap \Omega , a, \Omega )}{t^q} < \infty
$$
In particular, $\Omega$ satisfies a $p$-Upper Power Law and a $p$-Asymptotic Power Law  with exponent $q$.
\end{enumerate}
\end{theo}

\begin{rem}\label{genrem}
\rm
If $\Omega$ is convex, Theorem 2.1 in \cite{LLMTW} gives
$$
\omega_p (B(\xi , t) \cap \partial \Omega , a, \Omega ) \leq C t^q
$$
where $C$ is a positive constant depending on $n$, $p$ and $|a- \xi|$. Then the upper bounds in Theorem \ref{genthm} follow and we will only need to prove the lower bounds.
\end{rem}

The structure of the note is as follows: section \ref{majori} contains the adjusted comparison result (Lemma \ref{majorilem}). In section \ref{scale} we use Lemma \ref{majorilem} to deduce an intermediate $p$-harmonic measure estimate, in the range $1 < p < n$ (Lemma \ref{sqrtlem}). Then a rescaling trick provides a functional inequality for $\omega_p (J(t) , 0 , \B )$ in terms of $t$ (Lemma \ref{resclem}). Section \ref{proofs} contains the proof of Theorem \ref{mainth}. Finally, section \ref{gen} contains the suitable modifications to obtain Theorem \ref{genthm} from Theorem \ref{mainth}.

\section{Majorization}\label{majori}

Let $1 < p < \infty$ and $0 < t \leq 1/4$. In order to use $0$ as a pole of subsequent singular quasiradial functions, it will be convenient to perform the following technical reductions. By means of a rotation and a translation, we replace $0$ by the point $a = (0,...,0,1-t)$, the unit ball $\B$ by the ball $ B = B(a,1)$ and $J(t)$ by the cap  $C(t)$  given by
\begin{equation}\label{capdef}
C(t) = \{ x \in \partial B : \, < x-a, S-a> \, \geq \cos t  \}
\end{equation}
where $S = (0,...,0, -t)$.
Note that, if $n=2$, then  $C(t) = \{ i(1-t) + e^{i\theta}: \, |\theta + \pi/2 | \leq t \}  $.

Now, define
\begin{equation}\label{defG}
G = B \cap \{ z = (x,y)\in \R^n : \, y>0 \, , \, |z| > t   \}
\end{equation}
and note that $\partial G$ consists of three pieces:
\begin{eqnarray}
E_1 & = & \partial G \cap \{|z| = t \} \label{e1}\\
E_2 & = & \partial G \cap \{ y = 0 \} \label{e2}\\
E_3 & = &  \partial G \cap \partial B \label{e3}
\end{eqnarray}
The first result is a geometrical elementary estimate. Recall that for $z \in \R^n_+$ we denote by  $(r, \alpha )$ the polar coordinates of $z$, where $r = |z| $ and $\alpha$ is the latitude angle.

\begin{lem}\label{geomlem}
Let $0 < t \leq 1/4$ and $B, G$ as above. Suppose that $z\in E_3 $  has polar coordinates $(r, \alpha )$. Then $\sin \alpha \leq  r$.
\end{lem}

\begin{proof}
If $z = (x,y) \in E_3 $ then $|x|^2 + [ y - (1-t)]^2 = 1$ so
$$
r^2 = 2t - t^2 + 2(1-t)r \sin \alpha
$$
Then
\begin{equation}\label{polar}
\sin \alpha = \frac{r^2 - (2t- t^2)}{2(1-t)r} \leq \frac{r^2}{r} = r
\end{equation}
\end{proof}

Now, for $1 < p < \infty$, let $h = r^k f( \alpha)$ be the singular quasiradial $p$-harmonic function in $\R^n_+$ with pole at $0$, where $k= k(p,n)$. If $q = -k(p,n)$, recall that $ q \geq 1$ if $1 < p \leq n$ and $q \leq 1$ if $p \geq n$ (see Appendix).

\begin{lem}\label{majorilem}
Let $1 < p < \infty$, $0< t \leq 1/4$, $a = (0,...,0, 1-t) $, $B = B(a,1)$, $G$ as in \eqref{defG} and $q = -k(p,n)$. Let $h = r^{-q}f(\alpha)$ be the singular quasiradial $p$-harmonic function in $\R^n_+$ with pole at $0$, where $f$ satisfies \eqref{f1} and \eqref{f2}. There exist positive constants $C_1 , C_2$ depending only on $p$ and $n$ such that, for each $z \in G$,
\begin{alignat}{2}
 t^q h(z)  & \leq    C_1 \omega (z) + C_2 t^q  \, , \qquad   & p \geq n \vspace{0.12cm} \label{majorieq1}\\
 t^q h(z)  & \leq   C_1 \omega (z) + C_2 t^{\frac{q+1}{2}} \, , \qquad &   1 < p < n \label{majorieq2}
\end{alignat}
where $\omega (z) = \omega_p (C(t), z, B)$ if $z \in G$.
\end{lem}

\begin{proof}
We will apply the Comparison Principle (see \cite{HKM}, Theorem 7.6) in $G$ to the $p$-harmonic functions $v, w$,  where $v = t^q h$ and $ w = C_1 \omega + C_2 t^q $ if $p \geq n$ or $w = C_1 \omega + C_2 t^{\frac{q+1}{2}}$ if $1 < p < n$. It is then enough to check that, for suitable choices of the constants $C_1 $ and $C_2$,
$$
v \leq w \quad \text{on} \, \, \, \partial G = E_1 \cup E_2 \cup E_3
$$
where $E_1 , E_2, E_3$ are as in \eqref{e1}, \eqref{e2}  and \eqref{e3}. We split the argument into the three parts of the boundary and it is only in the analysis of $E_3$ that the cases $p\geq n$ and $1 < p < n$ need to be distinguished.

\begin{itemize}
\item [i)] Let $z = (x,y) \in E_1$.  Then $r = |z| = t$ and $v(z) = t^q h(z) = t^q r^{-q} f(\alpha) \leq f( \pi /2) = 1$. On the other hand, by the Carleson estimate for $p$-harmonic measure (see \cite{deb2}, Lemma 5.1 or \cite{HKM}, Lemma 11.21), there exists  $c = c(p,n)>0$ so that $\omega (z) \geq c$ for each $z \in E_1$. Then \eqref{majorieq1} and \eqref{majorieq2} hold for $z\in E_1$ with the choice $C_1 = c^{-1}$ and any (positive) choice of $C_2$. \vspace{0.12cm}
\item [ii)] Let $z = (x,y) \in E_2$. Then $r >0$ and $\alpha = 0$ so $f( \alpha ) = 0 = v(z) $. Then \eqref{majorieq1} and \eqref{majorieq2} hold for $z\in E_2$ whatever the choice of (positive) $C_1 , C_2$. \vspace{0.12cm}
\item [iii)] Fix $z = (x,y) \in E_3$ and let $r = |z|$ and $\alpha \in [0, \pi/2]$ the latitude of $z$, so $y = r \sin \alpha$. Then
$$
f(\alpha) = f(\alpha) - f(0) \leq M \alpha \leq \frac{\pi}{2} M \sin \alpha
$$
where $M = \max_{[0, \pi]} |f' |$ only depends on $p$ and $n$. Therefore, by Lemma \ref{geomlem}
\begin{equation}\label{vesti}
v(z) = t^q r^{-q} f( \alpha ) \leq  C t^q r^{-q} \sin \alpha  \leq C t^q r^{1-q}
\end{equation}
where $C = \pi M /2$. Now we consider separately the cases $p\geq n$ and $1 < p \leq  n$. Take first $p \geq n$ (then $0 < q \leq 1$). Since $r\leq 2$ on $E_3$ it follows from \eqref{vesti} that $v(z) \leq 2C t^q$, so \eqref{majorieq1} holds with $C_2 = 2C$. Suppose now that $1 < p \leq n $ (then $q \geq 1$). Since $r \geq (2t - t^2)^{1/2} \geq t^{1/2}$, it follows that $r^{1-q} \leq t^{\frac{1-q}{2}}$. Therefore, from \eqref{vesti},
$$
v(z) \leq C t^q r^{1-q} \leq C t^q t^{\frac{1-q}{2}} = C t^{\frac{q+1}{2}}
$$
and \eqref{majorieq2} holds with the choice $C_2 = C$.
\end{itemize}
The lemma follows from the Comparison Principle.
\end{proof}

\section{Rescaling}\label{scale}

In this section we will restrict to the range $1 < p < n$. We start by the following consequence of Lemma \ref{majorilem}, rephrased in the unit ball $\B$. Note that we are interested in the decay of $\omega_p ( J(t),0, \B )$, where $J(t)$ is given by \eqref{defJ}.

\begin{lem}\label{sqrtlem}
Let $1 < p < n$. There exist $0< t_0 \leq 1/4$ and $C>0$, only depending on $p$ and $n$ such that
\begin{equation}\label{sqrt}
\omega_p (J(t) , (1 - \sqrt t ,0,...,0) , \B ) \geq C \, t^{\frac{q}{2}}
\end{equation}
if $0 < t \leq t_0$, where $q = - k(p,n)$.
\end{lem}

\begin{proof}
Let $B = B(a,1)$ and $C(t)$ be as in Lemma  \ref{majorilem}. If $\hat{z} = (0,...0, \sqrt t )$, note that $h(\hat{z}) = t^{-\frac{q}{2}}$ and we get from Lemma  \ref{majorilem}
$$
t^q \, t^{-\frac{q}{2}} = t^{\frac{q}{2}} \leq C_1 \omega_p (C(t) , \hat{z} , B) + C_2 t^{\frac{q+1}{2}}
$$
Choosing $t_0 = \frac{1}{4C_2^2}$ it follows that
$$
\frac{1}{2}t^{\frac{q}{2}} \leq t^{\frac{q}{2}} (1 - C_2 t^{1/2}) \leq C_1 \omega_p (C(t) , \hat{z} , B)
$$
so
\begin{equation}\label{sqrt2}
\omega_p (C(t) , \hat{z} , B) \geq \frac{1}{2C_1} t^{\frac{q}{2}}
\end{equation}
Now \eqref{sqrt} is equivalent to \eqref{sqrt2}, by making use of the suitable translation and rotation.
\end{proof}

The following rescaling lemma is the key result of the section.

\begin{lem}\label{resclem}
Let $1 < p < n$. There exist $0 < t_0 \leq 1/4$ and $C>0$, only depending on $p$ and $n$ such that
\begin{equation}\label{resceq}
\omega_p (J(t), 0, \B) \geq C \, t^{\frac{q}{2}} \omega_p (J(\sqrt t), 0, \B )
\end{equation}
whenever $0 < t \leq t_0$.
\end{lem}

\begin{proof}
Let $B' = B(0, 1 - \sqrt t )$ and  $H(t) = (1 - \sqrt t )J(\sqrt{t} ) $. By the Markov property, Harnack's inequality and Lemma \ref{sqrtlem},
$$
\omega_p (J(t), 0, \B) \geq \omega_p ( H(t) , 0, B') \inf_{z\in H(t)}\omega_p (J(t), z, \B ) \geq C t^{\frac{q}{2}} \omega_p (H(t), 0, B')
$$
Now, by the dilation-invariance of the $p$-laplacian, $\omega_p (H(t), 0, B') = \omega_p (J ({\sqrt t}), 0, \B ) $ and the lemma follows.
\end{proof}

\section{Proof of Theorem \ref{mainth}}\label{proofs}
As mentioned in Remark \ref{mainrem}, we only need to prove the lower bounds in \eqref{mainth1} and \eqref{mainth2}. We split the argument into the cases $p \geq n$ and $1 < p < n$.

\subsection{The case $ \mathbf{p\geq n}$}
With the notation of Lemma \eqref{majorilem}, it is enough to prove that
\begin{equation}\label{lowest1}
\omega_p (C(t), a, B ) \geq b_1 \, t^q
\end{equation}
if $0 < t \leq t_0 $, where $t_0$ and $b_1$  depend only on $p$ and $n$. If $C_1 , C_2$ are as in \eqref{majorieq1}, choose $r_0 = r_0 (p,n)$ small enough so that $r_0^q \leq (2C_2 )^{-1}$.  Set $P = (0,...,0,r_0 )$. Then, if $0 < t \leq t_0 (p,n)$ we get from Lemma \ref{majorilem},
$$
t^q h(P) = t^q r_0^{-q} \leq C_1 \omega_p (C(t), P, B) + C_2 t^q
$$
so, from the choice of $r_0$, we obtain $\omega_p (C(t),P, B) \geq C_2 (C_1 )^{-1} t^q $ and \eqref{lowest1} follows from Harnack  inequality.

\subsection{The case $ \mathbf{1 < p < n}$}

Observe  that inequality \eqref{resceq} is equivalent to
\begin{equation*}\label{resceq2}
\omega_p (J(t^2) , 0, \B ) \geq C \, t^q \omega_p (J(t) , 0, \B )
\end{equation*}
for $0 < t \leq t_0 = t_{0}(p,n) < 1$. Also, we can assume hereafter that $ 0 < C = C(p,n) < 1$. Define
$$
\phi (t) = \frac{\omega_p (J(t), 0, \B )}{t^q}
$$
Then $\phi$ satisfies the functional inequality
\begin{equation}\label{phi}
\phi (t^2 ) \geq C \, \phi (t)
\end{equation}
if $0 < t \leq t_0$. Note that, by well known properties of $p$-harmonic measure (\cite{HKM}), $\phi $ is strictly positive. Furthermore, observe that if $t \in [t_0^2, t_0 ]$ then the monotonicity of $\omega_p (J(t), 0, \B)$ gives
$$
\phi (t) = \frac{\omega_p (J(t) , 0, \B )}{t^q} \geq \frac{\omega_p (J(t_0^2 ), 0, \B )}{t_0^{q}} = m >0
$$
where $m = m(p,n) >0$.

To complete the proof, we only need to show that the functional inequality \eqref{phi} implies the logarithmic estimate in \eqref{mainth2}, Theorem \ref{mainth}. This is the content of the following elementary lemma (note that $t_0$, $C$, $\beta$, $m$ and $b$ depend only on $p$ and $n$).

\begin{lem}\label{funcineqlem}
Let $0 < t_0 < 1$ and $0 < C < 1$. Suppose that $\phi : (0, t_0] \to (0, +\infty )$ satisfies the functional inequality \eqref{phi} and that $\displaystyle m = \inf_{[t_0^2, t_0]}\phi >0$. Then there exist $\beta = \beta (C) >0$ and $b = b(C, t_0, m) >0$ such that
\begin{equation}\label{}
\phi (t) \geq \frac{b}{\big ( \log (\frac{1}{t})  \big )^{\beta} }
\end{equation}
if $0 < t \leq t_0$.
\end{lem}
\begin{proof}
Set $\displaystyle \beta = \frac{\log \frac{1}{C}}{\log 2}$, so $C = 2^{-\beta}$. By iterating \eqref{phi} we get $\phi (t^{2^N}) \geq C^N \phi (t) = 2^{-\beta N } \phi (t) $. Now fix $t \in (0, t_0]$ and choose $N \in \N$ such that $t_0^{2^{N+1}} < t \leq t_0^{2^N} $ or, equivalently, $\displaystyle t_0^2 < t^{2^{-N}} \leq t_0$. Then
$$
\phi (t) = \phi ((t^{2^{-N}})^{2^N} ) \geq C^N \phi (t^{2^{-N}}) \geq m 2^{-\beta N}
$$
Since $\displaystyle 2^{N} \leq \frac{\log \frac{1}{t}}{\log \frac{1}{t_0}} $ we finally obtain
$$
\phi(t) \geq \frac{m \big ( \log \frac{1}{t_0} \big )^{\beta}}{\big ( \log \frac{1}{t}  \big )^{\beta} }
$$
and the result follows with $b = m \big ( \log \frac{1}{t_0}  \big )^{\beta} $.
\end{proof}

\section{Proof of Theorem \ref{genthm}}\label{gen}
Let $\Omega \subset \R^n$ be a domain, $a\in \Omega$, $\xi \in \partial \Omega$ and $q = -k(p,n)$.  According to Remark \ref{genrem}, only the lower bounds need to be proved.

\

According to the Interior Ball Condition we can choose $b\in \Omega$ such that $B(b,r) \subset \Omega$, where $r = |b- \xi |$. Take $0 < t \leq r/2$. From Harnack inequality applied to a suitable Harnack chain covering the segment $[a,b]$ we get
\begin{equation}\label{harnack}
\omega_p ( B(\xi, t) \cap \partial \Omega , a , \Omega ) \geq C \omega_p ( B(\xi ,t) \cap \partial \Omega , b , \Omega)
\end{equation}
for some positive constant $C$ possibly depending on $n$,$p$,$a$,$\xi$ and $r$. Then it is enough to estimate the right hand side of \eqref{harnack} from below. Since $\R^n \setminus \Omega$ contains a half-space  having $\xi$ on its boundary, the following Carleson estimate holds:
$$
\omega_p (B( \xi, t), x, \Omega) \geq C_0 >0
$$
for each $x\in B(\xi, t/2) \cap \partial B(b,r)$ , where $C_0$ only depends on $n$ and $p$ (see Lemma 11.21 in \cite{HKM}). From the Comparison Principle we get
$$
C_0^{-1}\omega_p (B(\xi, t) \cap \partial \Omega , x , \Omega )  \geq   \omega_p (B(\xi, t/2) \cap \partial B(b,r), x, B(b,r) )
$$
if $x\in B(b,r)$. Evaluating at $b$ and rescaling,
\begin{eqnarray*}
C_0^{-1}\omega_p (B(\xi, t) \cap \partial \Omega , b , \Omega ) & \geq &   \omega_p (B(\xi, t/2) \cap \partial B(b,r),  b, B(b,r) ) \\
& = &  \omega_p (B(\mathbf{1}, t/(2r)) \cap \partial \B , 0, \B )  \\
& \geq &  \omega_p (J(t/(2r)), 0 , \B)
\end{eqnarray*}
Then the lower estimates in Theorem \ref{genthm} follow from the corresponding lower estimates in \eqref{mainth1} and \eqref{mainth2} (Theorem \ref{mainth}).

\section{Appendix}
We include here a short argument showing that $k(p,n) \geq -1$ if $p \geq n$ and $k(p,n) \leq  -1 $ if $1 < p \leq  n$, without appealing to any monotonicity result. Let  $h = r^{k} f(\alpha )$ be the singular quasiradial $p$-harmonic function in $\R^n_+$ with pole at $0$, where $k = k(p,n)$ and $f$ satisfies \eqref{f1}, \eqref{f2}.

\

First of all, note that the singular quasiradial $n$-harmonic function in $\R^n_+$ with pole at $0$ is given by
$$
u(z) = \frac{y}{|z|^2}  = r^{-1} \sin \alpha \, , \qquad z=(x,y) \in \R^n_+
$$
where, as usual, $r = |z|$ and $\alpha \in [0, \pi/2 ]$ is the latitude angle. Substituting this expresion in \eqref{ode} it turns out, after some elementary computations, that $\triangle_p  u $ is a positive multiple of $p-n$, so $u$ is $p$-subharmonic if $p\geq n$ and $p$-superharmonic if $1 < p  \leq n$.

\

Suppose first that $p >n$. Consider the domain $\Omega = \R^n_+ \cap \{|z| > 1 \} $. Note that $u = h = 0 $ on  ($\partial \Omega \cap \{ z= (x,0) : x \in \R^{n-1} \} )  \cup \{ \infty \} $. Now, fix $C>0$ so that $C \sin \alpha \leq f( \alpha )$ for all $\alpha \in [0, \pi /2]$. Then $Cu \leq h$ on $\partial \Omega \cap \{ z : |z| = 1 \} $.  Since $u$ is $p$-subharmonic and $h$ is $p$-harmonic, it follows from  the Comparison Principle (\cite{HKM}, Thm. 7.6 ; note that the version there includes the case of unbounded domains) that $Cu \leq h$ in $\Omega$. In particular, evaluation at $(0,...,0,r)$ gives
$$
C r^{-1} \leq r^k
$$
for all $r>1$, implying that $k \geq -1$.

\

The case $1 < p < n$ is similar. Now a constant $D>0$ can be chosen so that $D \sin \alpha \geq f( \alpha )$ for each $\alpha \in [0, \pi /2]$ so, since $u$ is $p$-superharmonic and $h$ is $p$-harmonic, another application of the Comparison Principle gives $ D u \geq h$ in $\Omega$ and, again, evaluation at $(0,...,0,r)$ implies $k \leq -1$.

 \end{document}